\documentclass[psamsfonts,12pt]{amsart}
\usepackage{amssymb,mathrsfs}
\usepackage[left=1.25in,right=1.25in,top=1in,bottom=1.0in]{geometry}

\newtheorem{deff}{Definition}[section]
\newtheorem{lemma}[deff]{Lemma}
\newtheorem{theorem}[deff]{Theorem}

\newtheorem{fact}[deff]{Fact}

\newtheorem{em-example}[deff]{Example}
\newtheorem{em-def}[deff]{Definition}        % definition(auxiliary)
\newtheorem{em-remark}[deff]{Remark}         % remark(auxiliary)
\newtheorem{em-question}[deff]{Question}

\newtheorem{claim}{Claim}

\def\ker{\mathop{\rm ker}}

\def\P{{\mathbb P}}

\title{Convergent sequences in minimal groups} 

\author[D. Shakhmatov]{Dmitri Shakhmatov}
\address{Graduate School of Science and Engineering,
Division of Mathematics, Physics and Earth Sciences\\
Ehime University, Matsuyama 790-8577, Japan}
\email{dmitri@dpc.ehime-u.ac.jp}

\begin{document}
\begin{abstract} 
A Hausdorff topological group $G$ is {\em minimal\/} if every continuous isomorphism $f:G\to H$ between $G$ and a Hausdorff topological group $H$ is 
open. Clearly, every compact Hausdorff group is minimal.
It is well known that every infinite compact Hausdorff group contains a 
non-trivial convergent sequence.
We extend this result to minimal {\em abelian\/} groups by proving that every infinite minimal abelian group contains a non-trivial convergent sequence. 
Furthermore, we show that ``abelian'' is essential and cannot be dropped.   
Indeed, for every uncountable regular cardinal $\kappa$ we construct a Hausdorff group topology $\mathscr{T}_\kappa$ on the free group $F(\kappa)$ with $\kappa$ many generators having the following properties: 
\begin{itemize}
\item[(i)]
 $(F(\kappa),\mathscr{T}_\kappa)$ is a minimal group;
\item[(ii)]
 every subset of  $F(\kappa)$ of size less than $\kappa$ is $\mathscr{T}_\kappa$-discrete (and thus also $\mathscr{T}_\kappa$-closed); 
\item[(iii)]
 there are no non-trivial proper $\mathscr{T}_\kappa$-closed normal subgroups of $F(\kappa)$.
\end{itemize}
\noindent
In particular, all compact subsets of $(F(\kappa),\mathscr{T}_\kappa)$ are finite, and every Hausdorff quotient group of $(F(\kappa),\mathscr{T}_\kappa)$ is minimal (that is, $(F(\kappa),\mathscr{T}_\kappa)$ is {\em totally minimal\/}).
\end{abstract}



\thanks{The author was partially supported by the Grant-in-Aid for Scientific Research no.~19540092 by the Japan Society for the Promotion of Science (JSPS)}

\maketitle

We denote by 
$\mathbb N$
the set 
the set of natural numbers.

Let $X$ be a topological space.
A {\em convergent sequence in $X$ \/}
is a sequence $S=\{x_n:n\in\mathbb{N}\}$ of points of $X$ such that
there exists a point $x\in X$ (called the {\em limit of $S$\/}) so that
$S\setminus U$ is finite for every open subset $U$ of $X$ containing $x$.
(We also say that $S$ {\em converges to $x$\/}.) A sequence $S$ is {\em non-trivial\/} provided that the set $S$ is infinite. 

The identity element of a group $G$ is denoted by $1$. When $G$ is abelian, the additive notation is used, and so $1$ is replaced by the zero element $0$ of $G$.
\section{Results}
Our starting point in this manuscript is the following folklore fact.
\begin{fact}
\label{main:fact}
Every infinite compact Hausdorff group contains a non-trivial convergent sequence.
\end{fact}
This result is a consequence of the theorem of Ivanovski\u{\i} \cite{Iv} and Kuz'minov \cite{Ku} that every compact group is dyadic. We refer the reader to \cite{Shak} for the proof of Fact \ref{main:fact} based on Michael's selection theorem in the spirit of \cite{Uspenskii}.

In Fact \ref{main:fact} compactness cannot be weakened to pseudocompactness or countable compactness, even in the abelian case. Indeed, there exists an example (in ZFC) of a
pseudocompact 
abelian
group without
non-trivial convergent sequences \cite{Sirota}. Furthermore, there are
numerous consistent examples of countably compact abelian groups without non-trivial
convergent sequences, see
\cite{HJ, vD2, Malyhin1,  Tka, Dijkstra-vanMill, DT, Tomita2002,DS4}.
However, the existence of a countably compact group without non-trivial convergent sequences in ZFC alone remains a major open problem in the area \cite[Problem 22]{DS:survey}.

Another well-known generalization of compactness in the class of topological groups is related to the fundamental fact that the topology of a Hausdorff compact space $X$ is a minimal element in the set of all Hausdorff topologies on the set $X$.   
\begin{deff}
{\em 
A Hausdorff group topology $\mathscr{T}$ on  a group $G$ is called  {\em minimal\/}  provided that every Hausdorff group topology $\mathscr{T}'$ on $G$ such that $\mathscr{T}'\subseteq \mathscr{T}$ satisfies 
$\mathscr{T}' = \mathscr{T}$. In such a case the pair $(G,\mathscr{T})$ is called a {\em minimal group\/}.
}
\end{deff}

 The notion of a minimal group was introduced independently by Choquet (see Do\" \i tchi\-nov \cite{Do}) and Stephenson \cite{S}.
We refer the reader to \cite{DPS, D} for additional information on minimal groups.

While compactness cannot be replaced by pseudocompactness or countable compactness in the statement of Fact \ref{main:fact},
 our first result demonstrates that 
compactness can be weakened to minimality provided that the group in question is commutative.

\begin{theorem}
\label{theorem:1}
Every infinite minimal abelian group contains a non-trivial convergent sequence.
\end{theorem}
The particular version of Theorem \ref{theorem:1} for countably compact groups has been announced without proof on page 393 of \cite{DS:survey} (see the text preceding \cite[Problem 23]{DS:survey}).

As usual, we say that a group $G$ is {\em non-trivial\/} provided that $|G|\ge 2$.

Our second result shows that the word ``abelian'' in Theorem \ref{theorem:1} is essential and cannot be omitted. 

\begin{theorem}
\label{theorem:2}
For every 
uncountable
regular cardinal $\kappa$ there exists a Hausdorff group topology $\mathscr{T}_\kappa$ on the free group $F(\kappa)$ with $\kappa$ many generators having the following properties: 
\begin{itemize}
\item[(i)]
 $(F(\kappa),\mathscr{T}_\kappa)$ is a minimal group;
\item[(ii)]
 every subset of  $F(\kappa)$ of size less than $\kappa$ is $\mathscr{T}_\kappa$-discrete (and thus also $\mathscr{T}_\kappa$-closed); 
\item[(iii)]
 there are no non-trivial proper $\mathscr{T}_\kappa$-closed normal subgroups of $F(\kappa)$.
\end{itemize}
In particular,
\begin{itemize}
\item[(a)]
all compact subsets of $(F(\kappa),\mathscr{T}_\kappa)$ 
are
finite, and
\item[(b)] 
every Hausdorff quotient group of $(F(\kappa),\mathscr{T}_\kappa)$ is minimal; that is, $(F(\kappa),\mathscr{T}_\kappa)$ is totally minimal.
\end{itemize}
\end{theorem}

\section{Proof of Theorem \ref{theorem:1}}

\begin{lemma}
\label{lemma:0}
An infinite subgroup of a compact metric group has a non-trivial convergent sequence.  
\end{lemma}
\begin{proof}
Assume that $G$ is an infinite subgroup of a compact metric group $K$.
Then
$G$ cannot be discrete, and thus the identity $1$ of $G$ is a non-isolated point of $G$. Let $\{U_n:n\in\mathbb{N}\}$ be a decreasing local base at $1$.
By induction on $n\in\mathbb{N}$ choose $x_n\in U_n\setminus\{x_0,\dots,x_{n-1}\}$.
Then $\{x_n:n\in\mathbb{N}\}$ is
a non-trivial sequence in $G$ converging to $1$.
\end{proof}

\begin{deff}
\cite{P,S}
{\em A subgroup $G$ of a topological group $K$ is said to be {\em essential in $K$\/} provided that $G\cap H$ is a non-trivial subgroup of $K$ 
for every non-trivial closed normal subgroup $H$ of $K$. 
}
\end{deff}

If $K$ is abelian, then every subgroup $H$ of $K$ is normal, and so the word ``normal'' can be omitted in the above definition. 

The notion of an essential subgroup is a crucial ingredient of the so-called
``minimality criterion'', due to Prodanov and Stephenson \cite{P,S}, describing dense minimal subgroups of compact groups. 
\begin{fact} {\rm (\cite{P,S}; see also \cite{DP1,DPS})}
\label{minimality-criterion}
A dense subgroup $G$ of a Hausdorff compact group  $K$ is minimal if and only if $G$ is essential in $K$.
\end{fact}

The straightforward proof of the following lemma is left to the reader. 
\begin{lemma}
\label{lemma:3}
If $G$ is an essential subgroup of an abelian topological group $K$, then 
$$
K[p]=\{x\in K: px=0\}\subseteq G
$$ 
for every prime number $p$. 
\end{lemma}

\begin{lemma}
\label{lemma:1}
Let $I$ be an infinite set, $\{K_i: i\in I\}$ a family of non-trivial topological groups and $G$ 
an essential
subgroup of the product $K=\prod_{i\in I} K_i$.
Then $G$ has a non-trivial convergent sequence.
\end{lemma}
\begin{proof}
We identify each $K_i$ with the closed normal subgroup
$$
\{1\}\times\dots\times \{1\}\times K_i\times \{1\}\times\dots\times\{1\}
$$ 
of $K$, where $K_i$ occupies the $i$th place.
For each $i\in I$, 
use
essentiality of $G$ in $K$
to fix 
$g_i\in G\cap K_{i}$ with $g_{i}\not=1$. 
Since $K_{i}\cap K_{j}=\{1\}$ whenever $i,j\in I$ and $i\not=j$,
it follows that 
$\{g_i:i\in I\}$ is a a faithfully indexed family of elements of $G$.
Choosing a faithfully indexed subset $\{i_n:n\in\mathbb{N}\}$ of $I$,
we obtain an infinite sequence
$\{g_{i_n}:n\in \mathbb{N}\}$ of elements of $G$ converging to $1$.
\end{proof}

\begin{lemma}
\label{lemma:2}
An essential
subgroup of a non-trivial Hausdorff compact torsion-free abelian group contains a non-trivial convergent sequence.
\end{lemma}
\begin{proof}
Assume that $G$ is 
an essential
subgroup of a non-trivial compact torsion-free abelian group $K$.
Since $K$ is torsion-free, the Pontryagin dual of $K$ is divisible, and 
from \cite[Theorem 25.8]{HRoss} we conclude that
there exists a sequence of cardinals $\{\sigma_p:p\in\mathbb P\cup\{0\}\}$ such that
\begin{equation}
\label{compact:torsion:free:group}
K=\widehat{\mathbb Q}^{\sigma_0}\times\prod_{p\in\mathbb P}{\mathbb Z}_p^{\sigma_p}, 
\end{equation}
where $\widehat{\mathbb Q}$ denotes the Pontryagin dual of the discrete group
$\mathbb{Q}$ of rational numbers, ${\mathbb P}$ is the set of all prime numbers, and ${\mathbb Z}_p$ denotes the group of $p$-adic integers.

If 
the product \eqref{compact:torsion:free:group}
can be (re-)written as a product of infinitely many non-trivial topological groups,
then the conclusion of our lemma follows from 
Lemma \ref{lemma:1}.
In the remaining case $K$ is metrizable being a finite product of compact metric groups. Since $K$ is non-trivial and $G$ is essential in $K$, there exists 
$g\in G\cap K$
with $g\not=0$. Since $K$ is torsion-free, $g$ has an infinite order, and so $G$ is an infinite group.
Applying Lemma \ref{lemma:0}, we obtain a non-trivial convergent sequence in $G$.
\end{proof}

\begin{proof}[\bf Proof of Theorem \ref{theorem:1}]
Assume that $G$ is an infinite minimal abelian group.
Then its completion $K$ is a compact Hausdorff abelian group
\cite{PS} (see also \cite{DPS}).
Moreover, $G$ is essential in $K$ by Fact \ref{minimality-criterion}.
We consider two cases, depending on the size 
of the torsion part 
$$
t(K)=\{x\in K: nx=0
\mbox{ for some }
n\in\mathbb{N}\setminus\{0\}\}
$$
 of $K$.

\medskip
{\em Case 1\/}. {\sl $t(K)$ is uncountable\/}. 
Then 
the $p$-rank $r_p(K)$ of $K$ must be uncountable for some $p\in\P$. 
In particular, $K[p]$ is uncountable. Being a closed subgroup of the compact group $K$, the group $K[p]$ is compact.
Hence $K[p]$ contains a non-trivial convergent sequence by Fact \ref{main:fact}.
Finally, $K[p]\subseteq G$ by Lemma \ref{lemma:3}.
 
\medskip
{\em Case 2\/}. {\sl $t(K)$ is at most countable\/}.
Then $U=K\setminus(t(K)\setminus\{0\})$ is a $G_\delta$-subset of $K$ containing $0$. Therefore, there exists a closed $G_\delta$-subgroup $N$ of $K$  satisfying
$N\subseteq U$ (see, for example, \cite[Chapter II, Theorem 8.7]{HRoss} or \cite{Arh}). In particular, $N\cap t(K)=\{0\}$. This means that $N$ is torsion-free.

If $N\not=\{0\}$, then $N$ is a non-trivial compact abelian group. Since $G$ is essential in $K$, $G\cap N$ is essential in $N$. Since $N$ is torsion-free, Lemma 
\ref{lemma:2} yields 
that
$G\cap N$ (and thus $G$ as well) has a non-trivial convergent sequence. 

If $N=\{0\}$, then $\{0\}$ is a $G_\delta$-subset of $K$, and so $K$ is metrizable. 
Applying Lemma \ref{lemma:0}, we obtain a non-trivial convergent sequence in $G$.
\end{proof}

\section{Proof of Theorem \ref{theorem:2}}

The construction in this section is inspired by an old construction of the author \cite{Sh}.

Given a set $X$, the symbol $S(X)$ denotes the {\em symmetric group of $X$\/}, i.e., the set of all bijections of the set $X$ with the composition of maps as the group operation. We equip $S(X)$ with the {\em topology of pointwise convergence on $X$\/} whose base is given by the family 
$$
\mathscr{W}(X)=\{W(X,Z,\varphi): Z
\mbox{ is a finite subset of }
X
\mbox{ and }
\varphi: Z\to X
\mbox{ is an injection}
\},
$$
where 
\begin{equation}
\label{definition:of:W}
W(X,Z,\varphi)=\{f\in S(X): f\restriction_Z=\varphi\}.
\end{equation}

As usual, an ordinal $\alpha$ is considered to be the set consisting of all smaller ordinals; that is, $\alpha=\{\beta: \beta<\alpha\}$. In what follows, 
$F(\alpha)$ denotes the free group with the alphabet $\alpha$.
For special emphasis, 
we use $*_\alpha$ to denote the group operation of $F(\alpha)$
and $e_\alpha$ to denote the identity element of $F(\alpha)$.

Fix an uncountable regular cardinal $\kappa$. For $\gamma\in\kappa+1$ 
define
\begin{equation}
\label{definition:of:T}
T_\gamma=\{(\alpha,\beta)\in(\gamma\setminus\omega)\times\gamma: 
\beta<\alpha\}
\end{equation}
and
\begin{equation}
\label{definition:of:X}
X_\gamma=\bigcup_{\alpha\in\gamma}\{\alpha\}\times F(\alpha).
\end{equation}

For every $\gamma\in \kappa\setminus\omega$ we have 
$|T_{\gamma+1}|=|\gamma|$, 
so we can fix  an injection $j_\gamma: T_{\gamma+1}\to \gamma$.

\begin{claim}
\label{J:is:a:monomorphism}
The unique homomorphism $J_\gamma: F(T_{\gamma+1})\to F(\gamma)$ 
extending $j_\gamma$
is an injection.
\end{claim}

For each $\gamma\in \kappa\setminus\omega$, the family $H_\gamma$ of all bijections of $X_\gamma$ that move only finitely many elements of $X_\gamma$, is dense in $S(X_\gamma)$ and has size $|X_\gamma|=|\gamma|$, so we can fix an enumeration 
\begin{equation}
\label{definition:of:H} 
H_\gamma=\{h_{\gamma \beta}: \beta\in\gamma\}.
\end{equation}

For $(\alpha,\beta)\in T_\kappa$ define $f_{\alpha, \beta}\in S(X_\kappa)$ by
\begin{equation}
\label{definition:of:f}
f_{\alpha,\beta}(\gamma, g)=
\left\{ \begin{array}{ll}
h_{\alpha\beta}(\gamma, g), & \mbox{ for } \gamma\in\alpha \\
(\gamma, g *_\gamma j_\gamma(\alpha,\beta) ), & \mbox{ for } \gamma\in\kappa\setminus\alpha
\end{array} \right.
\hskip25pt \mbox{ for } (\gamma, g)\in X_\kappa.
\end{equation}

Define 
\begin{equation}
\label{definition:of:Y}
Y_\kappa=\{f_{\alpha,\beta}:(\alpha,\beta)\in T_\kappa\}\subseteq S(X_\kappa),
\end{equation}
and 
let $G_\kappa$ to be the subgroup of $S(X_\kappa)$ generated by $Y_\kappa$.
Define 
the map $\theta: T_{\kappa}\to Y_\kappa$ by 
\begin{equation}
\label{definition:of:theta}
\theta(\alpha,\beta)=f_{\alpha, \beta}
\ \mbox{ for }\ 
(\alpha,\beta)\in T_{\kappa},
\end{equation}
 and let
 $\Theta: F(T_{\kappa})\to G_\kappa$ be the unique homomorphism extending 
$\theta$.

\begin{claim}
\label{calculation:claim}
$\Theta(g)(\gamma,e_\gamma)=(\gamma, J_\gamma(g))$
whenever $\gamma\in\kappa\setminus\omega$ and $g\in F(T_\gamma)$.
\end{claim}
\begin{proof}
The conclusion of our claim obviously holds for the identity element of $F(T_\gamma)$, so we will assume that $g$ is not the identity of $F(T_\gamma)$.
Then there exist $n\in\mathbb{N}$, $\{(\alpha_k,\beta_k):k\le n\}\subseteq T_\gamma$
and $\{\varepsilon_k:k\le n\}\subseteq \{-1,1\}$ such that
$$
g=\prod_{k=0}^n (\alpha_k,\beta_k)^{\varepsilon_k}.
$$
Together with \eqref{definition:of:theta} this yields
\begin{align}
\label{theta}
\Theta(g)=\Theta\left(\prod_{k=0}^n (\alpha_k,\beta_k)^{\varepsilon_k}\right)&=
\prod_{k=0}^n \theta(\alpha_k,\beta_k)^{\varepsilon_k}\\
&=
\prod_{k=0}^n \left(f_{\alpha_k,\beta_k}\right)^{\varepsilon_k}
=
f_{\alpha_n,\beta_n}^{\varepsilon_n}\circ f_{\alpha_{n-1},\beta_{n-1}}^{\varepsilon_{n-1}}\circ\ldots
\circ f_{\alpha_0,\beta_0}^{\varepsilon_0}.
\notag
\end{align}
From \eqref{theta} and \eqref{definition:of:f} we get
\begin{align*}
\Theta(g)(\gamma,e_\gamma)&=
f_{\alpha_n,\beta_n}^{\varepsilon_n}\circ f_{\alpha_{n-1},\beta_{n-1}}^{\varepsilon_{n-1}}\circ\ldots
\circ f_{\alpha_0,\beta_0}^{\varepsilon_0}(\gamma,e_\gamma)
\\
&=
f_{\alpha_n,\beta_n}^{\varepsilon_n}\circ f_{\alpha_{n-1},\beta_{n-1}}^{\varepsilon_{n-1}}\circ\ldots
\circ f_{\alpha_1,\beta_1}^{\varepsilon_1}(\gamma,j_\gamma(\alpha_0,\beta_0)^{\varepsilon_0})
\\
& \dots 
\\
&=f_{\alpha_n,\beta_n}^{\varepsilon_n}\left(\gamma,\prod_{k=0}^{n-1} j_\gamma(\alpha_k,\beta_k)^{\varepsilon_k}\right)
\\
&=\left(\gamma, \prod_{k=0}^{n} j_\gamma(\alpha_k,\beta_k)^{\varepsilon_k}\right)
=
(\gamma,J_\gamma(g)).
\end{align*}
\end{proof}

\begin{claim}
\label{isomorphism:claim}
$\Theta: F(T_{\kappa})\to G_\kappa$ is an isomorphism.
\end{claim}
\begin{proof}
Since $Y_\kappa$ generates $G_\kappa$ and $\Theta(T_\kappa)=\theta(T_\kappa)=Y_\kappa$ 
by \eqref{definition:of:Y} and \eqref{definition:of:theta},
it follows that $\Theta$ is a surjection.

To prove that $\Theta$ is an injection, assume that 
 $g\in F(T_\kappa)$ and $\Theta(g)$ is the identity map of $S(X_\kappa)$.
From \eqref{definition:of:T} we have 
\begin{equation}
\label{union:of:F}
F(T_\kappa)=\bigcup_{\alpha\in\kappa\setminus\omega} F(T_\alpha),
\end{equation}
and so there exists some 
$\gamma\in\kappa\setminus\omega$ with $g\in F(T_\gamma)$.
Since $\Theta(g)$ is the identity map of $S(X_\kappa)$, from Claim \ref{calculation:claim} we get $J_\gamma(g)=e_\gamma$.
Then $g=1$ by 
Claim \ref{J:is:a:monomorphism}.
Therefore, $\ker\Theta=\{1\}$, and so $\Theta$ is an injection.
\end{proof}

\begin{claim}
\label{levels:are:discrete}
$\Theta(F(T_\gamma))$ is discrete for every $\gamma\in\kappa\setminus\omega$.
\end{claim}
\begin{proof}
For each $g\in F(T_\gamma)$ let $\varphi_g: \{(\gamma,e_\gamma)\}\to X_\kappa$
be the map defined by 
\begin{equation}
\label{definition:of:varphi}
\varphi_g(\gamma,e_\gamma)=(\gamma, J_\gamma(g)),
\end{equation}
 so that 
$W(X_\kappa, \{(\gamma,e_\gamma)\}, \varphi_g)\in\mathscr{W}$.

(i) For every $g\in F(T_\gamma)$
one has
$$
\Theta(g)(\gamma,e_\gamma)=(\gamma, J_\gamma(g))=\varphi_g(\gamma,e_\gamma)
$$
by Claim \ref{calculation:claim} and \eqref{definition:of:varphi},
which yields 
$$
\Theta(g)\in W(X_\kappa, \{(\gamma,e_\gamma)\}, \varphi_g).
$$

(ii) Suppose that $g_0,g_1\in F(T_\gamma)$ and $g_0\not=g_1$.
Then $J_\gamma(g_0)\not=J_\gamma(g_1)$ by Claim 
\ref {J:is:a:monomorphism},
which together with \eqref{definition:of:varphi} yields $\varphi_{g_0}(\gamma,e_\gamma)\not=\varphi_{g_1}(\gamma,e_\gamma)$,
and thus
$$
W(X_\kappa, \{(\gamma,e_\gamma)\}, \varphi_{g_0})\cap W(X_\kappa, \{(\gamma,e_\gamma)\}, \varphi_{g_1})=\emptyset.
$$

Since $\mathscr{W}$ is a base for the topology of $S(X_\kappa)$,
from (i) and (ii) we conclude that the family 
$$
\{W(X_\kappa, \{(\gamma,e_\gamma)\}, \varphi_g): g\in F(T_\gamma)\}
$$ 
witnesses that the set $\Theta(F(T_\gamma))$ is discrete.
\end{proof}

\begin{claim}
\label{discrete:claim}
If $D\subseteq G_\kappa$ and $|D|<\kappa$, then $D$ is discrete.
\end{claim}
\begin{proof}
Since $\Theta$ is a bijection by Claim \ref{isomorphism:claim},
$|\Theta^{-1}(D)|=|D|<\kappa$.
By \eqref{definition:of:T}, 
$F(T_\lambda)\subseteq F(T_\mu)$ whenever $\omega\le\lambda<\mu<\kappa$, 
so using \eqref{union:of:F} and regularity of $\kappa$ we can find 
$\gamma\in\kappa\setminus\omega$ such that 
$\Theta^{-1}(D)\subseteq F(T_\gamma)$ 
and so $D\subseteq \Theta(F(T_\gamma))$.
Now 
Claim \ref{levels:are:discrete} applies.
\end{proof}

\begin{claim}
\label{dense:claim}
$Y_\kappa$ is dense in $S(X_\kappa)$. In particular, $G_\kappa$ is dense in $S(X_\kappa)$.
\end{claim}
\begin{proof}
Let $Z$ be a finite subset of $X_\kappa$ and $\varphi: Z\to X_\kappa$ be an injection. 
From \eqref{definition:of:X}
we get
$X_\kappa=\bigcup_{\gamma\in\kappa} X_\gamma$,
and $X_\lambda\subseteq X_\mu$ whenever $\omega\le\lambda<\mu<\kappa$.
Since $\kappa$ is uncountable and
$Z\cup \varphi(Z)$ is a finite subset of $X_\kappa$,
we have
\begin{equation}
\label{definition:of:alpha}
Z\cup\varphi(Z)\subseteq X_\alpha
\end{equation}
for some $\alpha\in\kappa\setminus\omega$.
The bijection $\varphi$ between two finite subsets $Z$ and $\varphi(Z)$ of $X_\alpha$ can be extended to a bijection of the whole $X_\alpha$. 
Therefore,   
$W(X_\alpha,Z,\varphi)$ is a non-empty open subset of $S(X_\alpha)$.
Since $H_\alpha$ is dense in $S(X_\alpha)$, using \eqref{definition:of:H} 
we can find 
$\beta\in\alpha$ such that
$h_{\alpha\beta}\in W(X_\alpha,Z,\varphi)$. That is, 
$h_{\alpha\beta}(\gamma,g)=\varphi(\gamma,g)$  for every $(\gamma,g)\in Z$.
Combining this with \eqref{definition:of:alpha} and
\eqref{definition:of:f}, we conclude that 
$f_{\alpha,\beta}(\gamma, g)=h_{\alpha\beta}(\gamma,g)=\varphi(\gamma,g)$ whenever $(\gamma,g)\in Z$. Together with \eqref{definition:of:W} and \eqref{definition:of:Y} this yields 
$$
f_{\alpha,\beta}\in W(X_\kappa, Z,\varphi)\cap Y_\kappa \not=\emptyset.
$$
Since $\mathscr{W}(X_\kappa)$ is a base for the topology of $S(X_\kappa)$, 
we conclude now that $Y_\kappa$ is dense in $S(X_\kappa)$.
\end{proof}

\begin{proof}[\bf Proof of Theorem \ref{theorem:2}]
Since $|T_\kappa|=\kappa$, the groups $F(\kappa)$ and $F(T_\kappa)$ are 
isomorphic.
Combining this with Claim \ref{isomorphism:claim}, we conclude that 
the groups $F(\kappa)$ and $G_\kappa$ are isomorphic.
Let $\mathscr{T}_\kappa$ be the topology on $F(\kappa)$ obtained by transferring the subgroup topology that $G_\kappa$ inherits from $S(X_\kappa)$ via the isomorphism between $F(\kappa)$ and $G_\kappa$.

Every dense subgroup of an infinite symmetric group $S(X)$ is minimal and has 
no proper non-trivial closed normal subgroups \cite{Die:Sch}. Combining this with Claim \ref{dense:claim}, we get (i) and (iii).
Claim \ref{discrete:claim} yields (ii). 
Finally, (a) follows from (ii), and 
(b) follows from (iii).
\end{proof}

\end{document}